\documentclass[a4paper,11pt,leqno]{amsart}

\usepackage{amsmath}
\usepackage{amsthm}
\usepackage{amsfonts}
\usepackage{amssymb}

\usepackage{tikz}
\usepackage{tikz-cd}
\usepackage{enumerate}
\usetikzlibrary{calc}
\usepackage{colonequals}
\usepackage[all]{xy}
\usepackage{todonotes}

\newcommand{\inv}{^{-1}}
\newcommand{\id}{\operatorname{id}}

\newcommand{\cone}{\operatorname{cone}}
\newcommand{\diam}{\operatorname{diam}}
\newcommand{\eps}{\varepsilon}
\newcommand{\N}{\mathbb N}
\newcommand{\Z}{\mathbb Z}
\newcommand{\R}{\mathbb R}

\newcommand{\bS}{\mathbb S}

\newtheorem{theorem}{Theorem}[section]
\newtheorem*{theorem*}{Theorem}{\bf}{\it}
\newtheorem{proposition}[theorem]{Proposition}
\newtheorem*{proposition*}{Proposition}{\bf}{\it}
\newtheorem{lemma}[theorem]{Lemma}
\newtheorem*{lemma*}{Lemma}{\bf}{\it}
\newtheorem{corollary}[theorem]{Corollary}
\theoremstyle{definition}
\newtheorem{definition}[theorem]{Definition}
\newtheorem*{definition*}{Definition}
\theoremstyle{remark}
\newtheorem{remark}[theorem]{Remark}
\numberwithin{equation}{section}
\newtheorem{example}[theorem]{Example}
\numberwithin{equation}{section}

\usepackage{pdfpages}

\begin{document}

\title[Open and Discrete Maps with PL Branch Set Images]{Open and Discrete Maps with Piecewise Linear Branch Set Images are Piecewise Linear Maps}

\author[R. Luisto]{Rami Luisto}
\address{Department of Mathematical Analysis, Charles University, Sokolovsk\'{a} 
  83, 186 00 Prague 8, Czech Republic
  \and
  Department of Mathematics and Statistics, P.O. Box 35, FI-40014 University of Jyv\"askyl\"a, Finland
  \and
  Digital Workforce Services,
  Mechelininkatu 1 a,
  Helsinki, Finland.
}
\email{rami.luisto@gmail.com}

\author[E. Prywes]{Eden Prywes}
\address{
  Princeton University,
  Fine Hall, Washington Road,
  Princeton NJ 08544-1000 USA
}
\email{eprywes@princeton.edu}

\thanks{
  The first author was partially supported by a grant of the Finnish Academy of Science and Letters,
  the Academy of Finland
  (grant 288501 `\emph{Geometry of subRiemannian groups}')
  and by the European Research Council
  (ERC Starting Grant 713998 GeoMeG `\emph{Geometry of Metric Groups}').
}

\date{\today}

\maketitle

\begin{abstract}
  The image of the branch set of a PL branched cover between PL $n$-manifolds is a simplicial 
  $(n-2)$-complex. We demonstrate that the reverse implication also holds: an open and discrete map
  $f \colon \bS^n \to \bS^n$ with the image of the branch set contained in a simplicial $(n-2)$-complex
  is equivalent up to homeomorphism to a PL branched cover.
\end{abstract}

\section{Introduction}
\label{sec:intro}

A continuous mapping between topological spaces is said to be \emph{open} if the image
of every open set is open and \emph{discrete} if the preimages of points are discrete sets in the domain.
The canonical example of an open and discrete map is the winding map in the plane $w_p(z) = \frac{z^p}{|z|^{p-1}}$, $p \in \Z$, and the 
higher dimensional analogues, $w_p \times \id_{\R^k} \colon \R^{k+2} \to \R^{k+2}$.
An important subclass of open and discrete maps is that of quasiregular mappings.
A mapping $f \colon \R^n \to \R^n$ is \emph{$K$-quasiregular} for some $K \geq 1$
if $f \in W^{1,n}_{\text{loc}}(\R^n)$ and for almost every $x \in \R^n$,
\[
  \|Df\| \le K\det(Df),
\]
where $\|Df\|$ is the norm of the weak differential of $f$ and $\det(Df)$ is the Jacobian determinant of $f$ (see \cite{Rickman-book}).
By the Reshetnyak theorem quasiregular mappings are open and discrete (\cite{Reshetnyak67} or \cite[Section IV.5, p.\ 145]{Rickman-book})
and so open and discrete maps can be seen as generalizations of quasiregular mappings, see
e.g.\ \cite{LuistoPankka-Stoilow} for some further discussion.

We denote by $B_f$ the \emph{branch set} of $f$. This is the set of points where $f$ fails to be
a local homeomorphism. 

Continuous, open and discrete maps are sometimes referred to in the literature as \textit{branched coverings} (see e.g., \cite{Rickman-book} and \cite{MartioSrebro}).
This is due to the fact that if the map is proper, then its restriction to the complement of the preimage of the image of its branch set is a covering map.
This terminology is not standard in other fields. Specifically, in the study of piecewise linear covering maps the term is used to mean an open and discrete piecewise linear map with a simplicial branch set.  For this reason, we will refer to such maps as \textit{open and discrete maps}.

In dimension two the branch set of open and discrete maps is well understood.
By the classical Sto\"ilow theorem (see e.g.\ \cite{Stoilow} or \cite{LuistoPankka-Stoilow}) the branch set of 
an open and discrete map between planar domains is a discrete set. In higher dimensions
the \v{C}ernavskii-V\"ais\"al\"a theorem \cite{Vaisala} states that the branch set of an open and discrete map
between two $n$-manifolds has topological dimension of at most $n-2$. Note that the aforementioned winding map
$w_p \colon \R^n \to \R^n$ gives an extremal example as the branch set of $w_p$ is the $(n-2)$-dimensional
subspace
\begin{align*}
  \{ (0,0,x_3, \dots, x_n):(x_3,\dots,x_n) \in \R^{n-2} \}.
\end{align*}
On the other hand the \v{C}ernavskii-V\"ais\"al\"a result is not strict in all dimensions.
In Section \ref{sec:Nemesis} we describe a classical example by Church and Timourian of an open and discrete map from $\bS^5$ to $\bS^5$ with $\dim_{\mathcal{T}}(B_f) = 1$. It is currently not known
if such examples exist in lower dimensions.  For example, the Church-Hemmingsen conjecture asks if
there exists an open and discrete map in three dimensions with a branch set homeomorphic to a Cantor set (see
\cite{ChurchHemmingsen1} and \cite{AaltonenPankka}).
In general the structure of the branch set of an open and discrete map, or even a quasiregular mapping,
is not well understood but the topic garners great interest. In Heinonen's
ICM address, \cite[Section 3]{HeinonenICM}, he asked the following: 
\begin{flushleft}
  \emph{Can we describe the geometry and the topology of the allowable branch sets of quasiregular
    mappings between metric $n$-manifolds?}
\end{flushleft}

In the setting of piecewise linear (PL) branched covers between PL manifolds the \v{C}ernavskii-V\"ais\"al\"a
result is exact in the sense that the branch set is $(n-2)$-dimensional.
Furthermore, it is a simplicial subcomplex of the underlying PL structure and the branched cover is locally a composition of winding maps.
Even without an underlying PL structure of the mapping, we can in some situations identify that an open and discrete map between Euclidean domains is a winding map. Indeed, by the classical results of
Church and Hemmingsen \cite{ChurchHemmingsen1} and Martio, Rickman and V\"ais\"al\"a \cite{MRV1971},
if the image of the branch set of an open and discrete map $f\colon \Omega \to \R^n$ is contained
in an $(n-2)$-dimensional affine subset, then the mapping is locally topologically equivalent to a winding map.
Winding maps, in turn, admit locally a canonical PL-structure.

The following is the main theorem of this paper. For terminology
on simplicial complexes and cones we refer to Section \ref{sec:Preli}. 

\begin{theorem}\label{thm:maintheorem}
  Let $\Omega \subset \R^n$ be a domain and $f \colon \Omega \to \R^n$ be an open and discrete map.  
  Suppose that $f(B_f)$ is contained in a simplicial $(n-2)$-complex.
  Then $f$ is locally topologically equivalent to a piecewise linear map which is a cone
  of a lower-dimensional PL branched cover $g \colon \bS^{n-1} \to \bS^{n-1}$.
\end{theorem}
We also prove a global version of Theorem \ref{thm:maintheorem}.
\begin{theorem}\label{thm:GlobalMainTheorem}
  Let $f \colon \bS^n \to \bS^n$ be an open and discrete map such that $f(B_f)$ is contained in a simplicial $(n-2)$-complex.
  Then $f$ is topologically equivalent to a PL branched cover.
\end{theorem}
We formulate and prove our results in the topological setting, but a quasiregular version of the theorem can be acquired using similar methods (see Section \ref{sec:EdenMap}).  The proofs for these theorems will be shown in Section \ref{sec:PL}.  We also remark that Theorem 1.2 holds for general compact PL manifolds and the proof is similar.

This gives a partial answer to a question posed by Heinonen and Semmes in \cite[Question 28]{HeinonenSemmes}.
The previous two statements assume that $f(B_f)$ is contained in a simplicial $(n-2)$-complex.
Since the results are stated up to topological equivalence, Theorem \ref{thm:maintheorem} still applies when $f(B_f)$ is
contained in a set $X$ such that for each point in $X$ there exists a neighborhood $U$ and a homeomorphism
$\phi\colon U \to B(0,1)$ that sends $X \cap U$ to an $(n-2)$-simplicial complex.
Theorem \ref{thm:maintheorem} was shown by Martio and Srebro \cite{MartioSrebro} in dimension three.

A straightforward consequence of Theorem \ref{thm:maintheorem} is that when $f(B_f)$ is contained in a codimension two simplicial complex, the topological dimension of $f(B_f)$ must be exactly $(n-2)$.
However, there are many open and discrete maps for which the image of the branch set is complicated.
Indeed, Heinonen and Rickman construct a quasiregular map $f \colon \bS^3 \to \bS^3$ containing a wild Cantor set in the branch set.
The set $\bS^3 \setminus f(B_f)$ is not simply connected. So, as a Cantor set with a topologically nontrivial complement, the set
$f(B_f)$ cannot be contained in a codimension 2 simplicial complex
(see \cite{HeinonenRickman} and \cite{HeinonenRickman2}).
Here a wild Cantor set refers to any Cantor set $C$ in $\R^n$ such that there is no homeomorphism
$h \colon \R^n \to \R^n$ for which $h(C) \subset \R \times \{ 0 \}^{n-1}$.

We also note that the hypothesis of the PL structure must be made on the image of the
branch set and not on the branch set itself.  In Section \ref{sec:Nemesis} we present a classical
example due to Church and Timourian \cite{ChurchTimourian} of an open and discrete map whose branch set is a simplicial complex, but whose image is not.

A crucial step in the proof of Theorem \ref{thm:maintheorem} is showing that the boundaries
of so-called \emph{normal domains} of the mapping $f$ are $(n-1)$-manifolds.
In dimensions above three we need to study, not only the boundary of
a normal domain $U$, but also the boundaries of the $(n-1)$-dimensional normal domains of the restriction
$f|_{\partial U} \colon \partial U \to f \partial U$, and so forth continuing these restrictions to boundaries
of normal domains all the way down to dimension 1.

This added level of detail turns out to be natural in higher dimensions. The complexity of a
mapping is reflected in how many levels of normal domain boundaries are manifolds.
In general, the boundaries of normal domains can be manifolds when the open and discrete map in question
is not locally a cone of a lower dimensional map.
In dimension three,
Martio and Srebro \cite{MartioSrebro} proved that the boundaries of normal domains are manifolds exactly when
the open and discrete map in question is locally a \emph{path} of lower dimensional open and discrete maps (see Sections \ref{sec:Preli} and \ref{sec:ReverseImplication} for the terminology).

We extend this result to higher dimensions in Section \ref{sec:ReverseImplication}.
A crucial step in this proof is showing that, in Euclidean
spaces, codimension one manifold foliations of punctured domains are necessarily spherical under certain conditions (see
Section \ref{sec:HomotopiesOfHomotopies}).
In dimensions four and above,
the theorem also generalizes naturally to state that the more levels of
boundaries of lower dimensional normal domains are manifolds, the more the mapping displays path-like properties. 
A path of open and discrete maps with nonempty branch sets
increases the topological dimension of the resulting branch set by one.  So this result
gives lower bounds on the topological dimension of the branch set of an open and discrete map (see Section \ref{sec:ReverseImplication}).

The underlying motivation of this
paper is to better understand the connections between the behavior of an open and discrete map $f$ and 
the sets $B_f$ and $f(B_f)$.
From this point of view, we find the apparent ``duality" between the structure of the branch set
and the properties of lower dimensional normal domains very promising.

Finally, as an application of our results, we construct examples of quasiregular mappings in Section \ref{sec:EdenMap} in the form of the
following proposition.
\begin{proposition}\label{prop:projective}
  For each $n \in \N$ there exists a non-constant quasiregular mapping
  $f \colon \R^{2n} \to \mathbb{CP}^n$.
\end{proposition}
As mentioned above, a large motivation for the contemporary study of open and discrete maps comes from their subclass of quasiregular mappings.
Often quasiregular mappings in dimensions larger than $2$ are difficult to construct, but
it can oftentimes be easier to construct open and discrete maps. Thus Proposition \ref{prop:projective} demonstrates
that Theorem \ref{thm:maintheorem} can be applied in some cases to enhance
an open and discrete map into a quasiregular mapping.

\section{Preliminaries}
\label{sec:Preli}

We follow the conventions of \cite{Rickman-book} and
say that $U \subset X$ is a \emph{normal domain} for $f \colon X \to Y$ if $U$ is a precompact domain such that
\begin{align*}
  \partial f (U) = f (\partial U).
\end{align*}
A normal domain $U$ is \emph{a normal neighborhood} of $x \in U$ if
\begin{align*}
  \overline{U} \cap f \inv (\{ f(x) \}) = \{ x \}.
\end{align*}
By $U(x,f,r)$, we denote the component of the open set $f \inv (B_Y(f(x),r))$ containing $x$.
The existence of arbitrarily small normal neighborhoods 
is essential for the theory of open and discrete maps. The following lemma
guarantees the existence of normal domains, the proof can be found in \cite[Lemma I.4.9, p.19]{Rickman-book} (see also \cite[Lemma 5.1.]{Vaisala}).
\begin{lemma}\label{lemma:TopologicalNormalDomainLemma}
  Let $X$ and $Y$ be locally compact complete path-metric
  spaces and $f \colon X \to Y$ an open and discrete map.
  Then for every point $x \in X$ 
  there exists a radius $r_0 > 0$ such that
  $U(x,f,r)$ is a normal neighborhood of $x$ for
  any $r \in (0,r_0)$.
  Furthermore,
  \begin{align*}
      \lim_{r\to 0}\diam U(x,f,r) = 0.
  \end{align*}
\end{lemma}

The following \v{C}ernavskii-V\"ais\"al\"a theorem (see \cite{Vaisala}) is fundamental in the study of open and discrete maps.
\begin{theorem}\label{thm:CernavskiiVaisala}
  Let $X$ and $Y$ be $n$-dimensional manifolds.  If $f\colon X \to Y$ is an open and discrete map, then the topological dimension of $B_f, f(B_f)$ and $f^{-1}(f(B_f))$ is bounded above by $n-2$.  In particular, $B_f$, $f(B_f)$ and $f\inv (f (B_f))$ have no interior points and do not locally separate
  the spaces $X$ nor $Y$.
\end{theorem}

Another concept that we will use below is that of a \emph{cone}.
\begin{definition}
  Let $X$ be a topological space.
  \begin{enumerate}
  \item {
      The \emph{cone} of $X$ is the set $\cone(X) \colonequals (X\times [0,1])/(X\times \{0\})$.
    }
    \item {
      If $Y$ is another topological space,
      a \emph{cone map} $f\colon \cone(X) \to \cone(Y)$ is a continuous map such that $f([(x,t)]) = [(h(x),t)]$ for some $h \colon X \to Y$ and for all $t \in [0,1]$.
      Note that a mapping $g \colon X \to Y$ induces a canonical cone map $\cone(X) \to \cone(Y)$, $[(x,t)] \mapsto [(g(x),t)]$
      which we will denote by $\cone(g)$.
      }
  \item {
      The \emph{suspension} of $X$, denoted $S(X)$, is the disjoint union of two copies of $\cone(X)$ glued together by the identity at $X\times \{1\}$.
    }
  
    \item {
      The \emph{suspension map of $f$}, denoted $S(f) \colon S(X) \to S(Y)$, is defined in a similar manner as the cone map.
    }
  \end{enumerate}
\end{definition}
Note that $\cone(\bS^k)$ is homeomorphic to the closed $(k+1)$-ball, and
$S(\bS^k)$ is homeomorphic to $\bS^{k+1}$.

\begin{definition}
  A mapping $f\colon X \to Y$ is \emph{topologically equivalent} to $g \colon X' \to Y'$ if there exists homeomorphisms $\phi$ and $\psi$ such that
  \[
    f = \psi^{-1} \circ g \circ \phi.
  \]
  In other words the following diagram commutes:
  \[
    \begin{tikzcd}
      X \arrow{r}{f} \arrow[swap]{d}{\phi} & Y \arrow{d}{\psi} \\
      X'  \arrow{r}{g} & Y'
    \end{tikzcd}.
  \]
  
\end{definition}

\subsection{Simplicial complexes and PL-structures}

We largely follow \cite{rourkesanderson} in our notation and terminology. 
We list some of the basic definitions and concepts in this section for the sake of completeness.
\begin{definition}
  Let $\{v_0,\dots,v_k\} \subset \R^n$ be a finite set of points not contained in any $(k-1)$-dimensional affine subset.  The $k$-\emph{simplex} $D$ is defined as
  \[
    D = \left\{ \sum_{i=1}^k \lambda_iv_i : \sum_{i=1}^k\lambda_i = 1, \lambda_i \ge 0\right\}.
  \]
  We say $D$ is \emph{spanned} by $\{v_1,\dots,v_k\}$.
\end{definition}
A \emph{face} of a simplex $D$ is a simplex spanned by a subset of the vertices that span $D$.
\begin{definition}
  A \emph{simplicial complex} $X$ is a locally finite collection of simplices such that
  \begin{enumerate}
  \item if $D_1 \in X$ and $D_2 $ is a face of $D_1$, then $D_2 \in X$, and
    
  \item if $D_1,D_2 \in X$, then $D_1\cap D_2$ is a face of both $D_1$ and $D_2$.
    
  \end{enumerate}
  The simplicial complex $X$ is \emph{$k$-dimensional} if the highest degree simplex in $X$ is a $k$-simplex.
\end{definition}
We will often consider $X$ as a subset of $\R^n$.  In this case we tacitly identify $X$ with the union of the simplices contained in $X$.

\begin{definition}
  Let $\Omega \subset \R^n$ be a domain. A mapping $f \colon \Omega \to \R^n$ is piecewise
  linear if there exists a simplicial complex $X = \Omega$ such that $f$ is linear on each $n$-simplex in $X$.
\end{definition}

\subsection{Algebraic topology}

We refer to \cite{Hatcher} for basic definitions and theory of homotopy and homology.
We denote the homotopy groups and the singular homology groups of a space $X$ by $\pi_k(X)$ and $H_k(X)$,
respectively, for $k \in \N$.
A closed $n$-manifold $M$ is said to be a \emph{homology sphere}
if $H_0(M) = H_n(M) = \Z$ and $H_k(M) = 0$ for all $k \neq 0,n$.

A homology sphere need not be a sphere. The canonical example
of a nontrivial homology sphere
is the so-called \emph{Poincar\'{e} homology sphere}, defined
by gluing the opposing faces of a solid dodecahedron together with
a twist (see e.g.\ \cite{Cannon} and \cite{KirbyScharlemann}). 
We will denote the Poincar\'{e} homology sphere by $P$
and note that even though the suspension $S(P)$ of $P$ is not
a manifold, the double suspension $S^2(P)$ of $P$ is homeomorphic
to $\bS^5$
(see again e.g.\ \cite{Cannon} and \cite{KirbyScharlemann}). 

An important result for us is the following celebrated theorem that is an immediate
corollary of the Hurewicz isomorphism theorem \cite[Theorem 4.32]{Hatcher}
combined with the generalized Poincar\'e conjecture (see \cite{smale}, \cite{freedman}, and \cite{kleinerlott}).
\begin{theorem}\label{thm:WhiteheadHomology}
  If $M$ is a simply connected homology sphere, then $M$ is homeomorphic to the $n$-dimensional sphere $\bS^n$.
\end{theorem}

\subsection{The double suspension of the cover $\bS^3 \to P$.}
\label{sec:Nemesis}

To contrast our results and underline the necessity of the more technical
arguments, we recall in this section a classical open and discrete map from $\bS^5$ to $\bS^5$ constructed
by Church and Timourian \cite{ChurchTimourian} with complicated
branch behavior. This example shares many
of the properties of open and discrete maps with $f(B_f)$ contained in an $(n-2)$-simplicial
complex, but it is not a PL mapping. For further discussion on this map see e.g.\
\cite{AaltonenPankka}.

We note first that the Poincar\'{e} homology sphere can be equivalently defined
as a quotient of $\bS^3$ under a group action of order 120 (see \cite{KirbyScharlemann}).
The mapping $f \colon \bS^3 \to P$ induced by the group action is a covering map,
and since $\bS^3$ is simply connected we see that $\bS^3$ is
the universal cover of the Poincar\'{e} homology sphere $P$.
As a covering map $f$ has an empty branch set but the suspension of $f$,
$S(f) \colon S(\bS^3) \to S(P)$,
has a branch set equal to the two suspension points. By definition of the cone of a map,
the preimage of either suspension point $P \times\{0\}$ or $P \times \{1\}$ is a point and
the preimage of any other point is a discrete set of $120$ points. Thus the double suspension of $f$,
\begin{align*}
  S^2(f) \colon S^2(\bS^3) \cong \bS^5 \to S^2(P) \cong \bS^5
\end{align*}
is an open and discrete map between $5$-spheres and has a branch set equal to the suspension of the two branch points of $S(f)$. 
Thus the branch set $B_{S^2(f)}$ is PL-equivalent to $\bS^1$ and so we see that $S^2(f)$
is an open and discrete map between two spheres with a branch set of codimension four.

The image of the branch set $B_{S^2(f)}$ is complicated since its
complement has a fundamental group of 120 elements. Furthermore even though the branch set is PL-equivalent to $\bS^1$, the image
of the branch set is not PL-equivalent to a simplicial complex even though it is a Jordan curve
in $\bS^5$. Thus the map $S^2(f)$ does not satisfy the hypothesis of our main theorem.

We also remark for future comparison that for $S^2(f)$ the 
boundaries of normal neighborhoods $U(x_0,f,r)$, where
$x_0$ is one of the two suspension points of the second suspension,
are homeomorphic to $S(P)$.  
This means that the suspension of the Poincar\'{e} homology sphere foliates a punctured neighborhood
of a point in $\R^5$, but the simply connected space $S(P)$ with homology groups of a sphere is not a manifold.

\section{Boundary of a normal domain}
\label{sec:Bdry}

In this section we show that for an open and discrete map
$f \colon \Omega \to \R^n$ with $f(B_f)$ contained in a simplicial $(n-2)$-complex,
the boundaries of sufficiently small normal domains are homeomorphic to a sphere.
The main step of the proof takes the form of an inductive argument where in the
inductive step we restrict an open and discrete map to the boundary of a small normal
domain and study the new map between the lower dimensional spaces.
Since we do not a priori know that the boundary of a normal domain is a manifold,
many of the results in this section are proved in a more general setting where the domain of the mapping is not assumed to be a manifold.

We begin with a few preliminary results on the behavior of $f$ on the boundary
of a normal domain. The following Lemma \ref{lemma:MappingAtTheBoundary}
is known to the experts in the field (see e.g.\ \cite{MartioSrebro})
but we give a short proof for the convenience of the reader.
\begin{lemma}\label{lemma:MappingAtTheBoundary}
  Let $X$ be a locally compact and complete metric space and $f \colon X \to \R^n$ an open and discrete map.
  Fix $x_0 \in X$ and let $r_0>0$ be such that $U_r \colonequals U(x_0,f,r)$ is a normal neighborhood
  of $x_0$ for all $r \leq r_0$. Then the restriction 
  \begin{align*}
    f|_{\partial U_r} \colon \partial U_r \to \partial B(f(x_0),r)
  \end{align*}
  is an open and discrete map for all $r<r_0$.
\end{lemma}
\begin{proof}
  The restriction is clearly continuous and discrete, so it suffices to show that
  it is an open map. 
  Let $V \subset \partial U_r$ be a relatively open set and suppose $y = f(x_1) \in f(V)$, where $x_1 \in V$. Additionally, suppose that $\{x_1,\dots,x_k\} = f^{-1}(y)$.  For $\delta > 0$ let $N_\delta (y) = B(y,\delta)\cap \partial B(f(x_0),r)$ and for $\epsilon > 0$ let $N_\epsilon(x_i) = B(x_i,\epsilon) \cap \partial U_r$.
  
  Fix $\epsilon > 0$ so that $N_\epsilon(x_i) \cap N_\epsilon(x_j) = \emptyset$ for $i \ne j$ and $N_\epsilon(x_1) \subset V$. By \cite[Lemma 5.15]{bonkmeyer}, there exists a $\delta >0$ such that 
  \begin{align*}
      f^{-1}(B(y,\delta)) \subset \cup_{i=1}^k B(x_i,\epsilon).
  \end{align*}
  Let $y' \in N_\delta(y)$.  There exists a path $\gamma$ connecting $y$ to $y'$ in $N_\delta(y)$.  By the path-lifting properties of open and discrete maps, (see e.g.\ \cite[Chapter II.3]{Rickman-book} for the Euclidean setting or \cite{Luisto-Characterization} for a general case), $\gamma$ can be lifted to paths $\gamma_1,\dots,\gamma_k$ each contained in $N_\epsilon(x_i)$.  The end point of each lift $x_i'$ maps to $y'$.  So $N_\delta(y) \subset f(N_\epsilon(x_1)) \subset f(V)$, which means that $f(V)$ is open.
  
\end{proof}

We will repeatedly choose suitably small normal neighborhoods for points in the domain. For
clarity we formulate this selection as the following lemma.
\begin{lemma}\label{lemma:LemmaX}
  Let $X$ be a locally connected, locally compact and complete metric space and $f \colon X \to \R^n$ an open and discrete map.
  Then for every $x \in X$ there exists a radius $r(x,f) > 0$ such that for all $r < r(x,f)$,
  $U(x,f,r)$ is a normal neighborhood of $x$.

  Furthermore if $f(B_f)$ is contained in an $(n-2)$-simplicial complex we may assume that 
  $f(B_f) \cap f (\partial U(x,f,r)) = f(B_f) \cap \partial B(f(x),r)$
  is contained in an $(n-3)$-simplicial complex (up to a global homeomorphism) for all $r < r(x,f)$.
\end{lemma}

\begin{proof}
The first assertion follows from Lemma \ref{lemma:TopologicalNormalDomainLemma}.

To prove the second assertion we first note that since $\partial B(f(x),r)$ is homeomorphic to $S^{n-1}$, there exists a finite simplicial decomposition so that $\partial B(f(x),r)$ becomes a finite $(n-1)$-simplicial complex.
We may subdivide the simplicial structures of the simplicial sets that contains $f(B_f)$ and $\partial B(f(x),r)$ so that their intersection forms a simplicial structure as well.

It remains to show that the dimension of the intersected structures is no greater than $n-3$. It cannot be larger than $n-2$ by our hypothesis on $f(B_f)$.
By the local finiteness of the complex that contains $f(B_f)$ we may choose $r(x,f)$ small enough so that the complex that contains $f(B_f)$ is not contained in $B(f(x_0,r))$ for all $r < r(x,f)$.  This gives that the dimension of the intersection is no greater than $n-3$.
\end{proof}

\subsection{Radial properties of the mapping $f$}

In the following arguments we need a consistent
way of describing boundaries of normal domains of mappings which are themselves
restrictions of ambient mappings to boundaries of normal domains. To this
end we define nested collections of lower dimensional normal domains.

\begin{definition}
  Let $\Omega \subset \R^n$ be a domain and $f \colon \Omega \to \R^n$ an open and discrete map.
  Denote by 
  $\mathcal{U}_{n-1}$ the collection of boundaries of normal
  domains $U(x,f,r) \subset \Omega$ with $r < r(x,f)$ as in Lemma \ref{lemma:LemmaX}.
  For $k = n-1, \ldots, 2$
  we similarly define $\mathcal{U}_{k-1}$ to be the collection of boundaries of normal
  domains $U(x,f|_{V},r) \subset V$, $V \in \mathcal{U}_{k}$, with $r < r(x,f|_{V})$ as in Lemma \ref{lemma:LemmaX}.
  We call these collections as \emph{lower dimensional normal domains}.
  
  By Lemma \ref{lemma:LemmaX}, in the case where $f(B_f)$ is contained in an $(n-2)$-simplicial complex we may assume that
  for given $1 \le k \le n-1$ and
  $V\in \mathcal{U}_k$ that the set $f(B_f) \cap f(\partial U(x,f|_{V},r))$
  is contained, up to a homeomorphism, in an $(n-3)$-simplicial complex for $r < r(x,f|_{V})$.
\end{definition}

\begin{lemma}\label{lemma:ImageIsASphere}
  Let $\Omega \subset \R^n$ be a domain and $f \colon \Omega \to \R^n$ be an open and discrete map
  with $f(B_f)$ contained in an $(n-2)$-simplicial complex. 
  Then for any $k = n-1,\ldots,1$ and $V \in \mathcal{U}_k$, $f(V)$ is homeomorphic to a sphere.
\end{lemma}
\begin{proof}
  By using an inductive argument we see that it suffices to study the
  case where $f(V) \subset f(U)$ with $U \in \mathcal{U}_{k+1}$ and
  $f(U)$ is a $(k+1)$-sphere.
  In this setting we may apply Lemma \ref{lemma:LemmaX} to see that $V = \partial W$, for a normal neighborhood $W$.  Additionally, $f(W)$ is homeomorphic to a ball of dimension $k+1$.  Since $W$ is a normal neighborhood, the set $f(V)$ is homeomorphic to a sphere of dimension $k$.
\end{proof}

We now prove the main proposition needed for the proof of Theorem \ref{thm:maintheorem}. 
It captures the fact that for open and discrete maps with $f(B_f)$ contained in an $(n-2)$-simplicial complex,
the branching should occur ``tangentially", i.e., inside the boundaries of normal domains.
Some of the steps of the proof are described in Figure \ref{fig:RadialLifts}.
\begin{proposition}\label{prop:UniqueLiftsTopDim} 
  Let $\Omega \subset \R^n$ be a domain and $f \colon \Omega \to \R^n$ be an open and discrete map such that
  $f(B_f)$ is contained in an $(n-2)$-simplicial complex. Then for any
  $x_0 \in \Omega$, there exists a sufficiently small $r < r(x_0,f)$ so that for $v \in \bS^{n-1}$, the
  path
  \begin{align*}
    \beta \colon [0,r] \to \overline{B}(f(x_0),r),
    \quad
    \beta(t) = (r-t) v + f(x_0)
  \end{align*}
  has a unique lift starting from any point
  $z_0 \in \overline{U}(x_0,f,r) \cap f \inv (\{ \beta (0) \})$.
\end{proposition}
\begin{proof}
  Choose $r$ small enough so that $f(B_f) \cap B(f(x_0),r)$ is contained in a codimension-2 radial set.  That is, there exists an $(n-2)$-simplicial complex $D$ containing $f(B_f)$, such that $D \cap B(x_0,r_1) = \frac{r_1}{r_2} (D \cap B(x_0,r_2))$.
  
  Suppose towards contradiction that the claim is false.
  Then there exist two different lifts of $\beta$, say
  $\alpha_1, \alpha_2 \colon [0,r] \to \overline{U}(x_0,f,r)$ satisfying,
  \begin{align*}
    \alpha_1(0) = \alpha_2(0) = z_0 \quad \text{and} \quad \alpha_1(s_0) \neq \alpha_2(s_0),
  \end{align*}
  for some $s_0 \in (0,r)$. Set
  \begin{align*}
    t_0 = \inf \{ t \in [0,r] \mid \alpha_1(t) \ne \alpha_2(t) \}.
  \end{align*}
  So $\alpha_1(t) = \alpha_2(t)$ for all $t \in [0,t_0]$, but for $s \in (t_0,t_0 + \epsilon)$ for small $\epsilon$, $\alpha_1(s) \ne \alpha_2(s)$.
  Without loss of generality we may assume that $t_0 = 0$ and that 
  \begin{align*}
    \alpha_1(t_0) = \alpha_2(t_0) = z_0.
  \end{align*}
  (see top part of Figure \ref{fig:RadialLifts}).
  
  Fix a radius
  $R < r(z_0,f)$ such that $\overline{B}(f(z_0),R) \subset B(f(x_0),r(x_0,f))$ (see middle part of Figure \ref{fig:RadialLifts}).
  Let $s_0 \in (t_0, t_0 + \epsilon)$, we may assume that $\epsilon$ is sufficiently small so that $\beta(s_0) \in B(f(z_0),R)$.
  We now let $U(\alpha_1(s_0))$ and $U(\alpha_2(s_0))$ be normal neighborhoods of $\alpha_1(s_0)$ and $\alpha_2(s_0)$ respectively.
  Let $\zeta$ be a line segment that has one endpoint at $\beta(s_0)$ and intersects $f(B_f)$ only at $\beta(s_0)$.  Additionally, suppose that $\zeta$ is small so that 
  \begin{align*}
    \zeta \subset fU(\alpha_1(s_0))\cap fU(\alpha_2(s_0)).
  \end{align*}
  Since everything is contained in the image of normal neighborhoods we can lift $\zeta$ to $\gamma_1 \subset U(\alpha_1(s_0))$ and $\gamma_2 \subset U(\alpha_2(s_0))$
  from the points $\alpha_1(s_0)$ and $\alpha_2(s_0)$, respectively 
  -- note though that these lifts might not be unique.  Let $\gamma_3$ be a path connecting $\gamma_1$ and $\gamma_2$ that lies outside of $f^{-1}(f(B_f))$.  The path $f(\gamma_1\cup\gamma_2\cup\gamma_3)$ will be a loop based at $\beta(s_0)$ that consists of a line segment and a loop.  The loop will lie outside of $f(B_f)$ (see bottom part of Figure \ref{fig:RadialLifts}).

By Lemma \ref{lemma:TopologicalNormalDomainLemma}, for any sufficiently small radius $\rho$, there exists a normal neighborhood of $z_0$ that maps onto $B(f(z_0),\rho)$.
So we may choose a sufficiently small normal neighborhood
of $z_0$, which we denote by $V$, for which image of the branch set in $V$ is also radial
with respect to $f(z_0)$. The point $\beta(s_0)$ lies on a path
between $f(z_0)$ and $f(x_0)$ so the image of the branch set will be radial at $\beta(s_0)$
with respect to small enough normal neighborhoods as well.
Indeed, for any $w \in B(f(z_0),R)$ that does not lie in the simplicial complex containing $f(B_f)$, the line segments
$[w,f(z_0)]$ and $[w,f(x_0)]$ will not belong to the simplicial complex that contains the image of the branch set (not including $f(z_0)$ and $f(x_0)$).
Additionally, for each $w' \in [w,f(z_0)]$, the line $[w',f(x_0)]$ will not be in the simplicial complex containing $f(B_f)$ and
so we conclude that the simplicial complex containing $f(B_f)$ does not intersect $[w,\beta(s_0)]$ except at $\beta(s_0)$.

Define a homotopy that consists of the straight line from each point in $f(\gamma_1\cup\gamma_2\cup\gamma_3)$ to $\beta(s_0)$.
Due to the argument mentioned in the previous paragraph, we see that the
homotopy will take $f(\gamma_1\cup\gamma_2\cup\gamma_3)$  to an arbitrarily
small neighborhood of $\beta(s_0)$ without intersecting $f(B_f)$.  Additionally, the end loop will be contained in the image of the normal neighborhoods, $U(\alpha_1(s_0))$ and $U(\alpha_2(s_0))$.

The homotopy can be lifted uniquely since it avoids $f(B_f)$.  The resulting loop of the homotopy will be lifted separately to the normal neighborhoods of $\alpha_1(s_0)$ and $\alpha_2(s_0)$.  This gives a homotopy from a connected curve to two disconnected loops, which is a contradiction.
\end{proof}

\begin{figure} 
  \centering
  \includegraphics[width=1\textwidth]{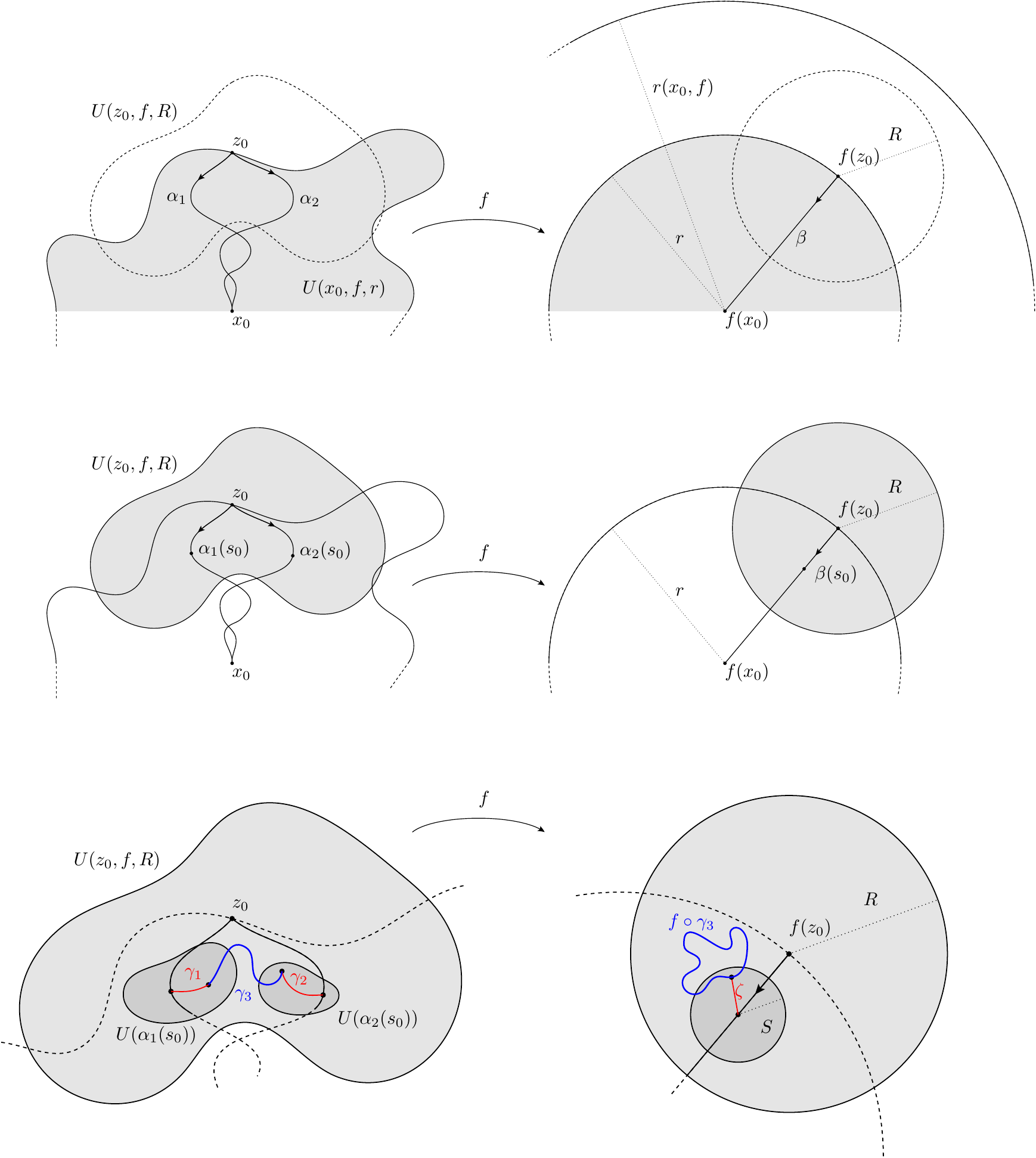}
  \caption{Showing that radial lifts are unique.}
  \label{fig:RadialLifts}
\end{figure}

The previous proposition allows us to uniquely lift radial paths in normal neighborhoods.
We would also like to be able to lift radial paths uniquely in lower dimensional normal neighborhoods $U \subset V$ with $V \in \mathcal{U}_k$ for any $k$.

Let $f \colon \Omega \to \R^n$ be an open and discrete map such that $f(B_f)$ is contained in an $(n-2)$-simplicial complex.
Let $x_0 \in V \in \mathcal{U}_k$.  By Lemma \ref{lemma:ImageIsASphere}, $f(V) \cong \bS^k$.  Up to homeomorphism we can assume that $f(V)$ minus a point maps to a $k$-dimensional plane.  In this case $f|_V$ will have the image of the branch set contained in a $(k-2)$-simplicial complex.

\begin{proposition}\label{prop:UniqueLifts}
Let $\Omega \subset \R^n$ be a domain and $f \colon \Omega \to \R^n$ be an open and discrete map such that
  $f(B_f)$ is contained in an $(n-2)$-simplicial complex.
  Let $V \in \mathcal{U}_k$.
For any
  $x_0 \in V$, there exists a sufficiently small $r < r(x_0,f|_v)$ so that for $v \in \bS^{n-1}$, the
  path
  \begin{align*}
    \beta \colon [0,r] \to \overline{B}(f(x_0),r),
    \quad
    \beta(t) = (r-t) v
  \end{align*}
  has a unique lift starting from any point
  $z_0 \in \overline{U}(x_0,f|_V,r) \cap f|_V^{-1} (\{ \beta (0) \})$ that is contained in $V$.
\end{proposition}
\begin{proof}
By Proposition \ref{prop:UniqueLiftsTopDim} we know that $\beta$ has a unique lift in $\Omega$ starting from any preimage of $\beta(0)$.
Thus we only need to show that such a lift is contained $V$.
By assumption, $x_0$ is a preimage of $\beta(0)$ in $V$ and the lift of $\beta$ under $f|_V$ starting from this preimage is contained in $V$. 
\end{proof}

\begin{proposition}\label{prop:HomeomorphismFoliation}
  Let $\Omega \subset \R^n$ be a domain and $f \colon \Omega \to \R^n$ an open and discrete map
  with $f(B_f)$ contained in an $(n-2)$-simplicial complex. 
  Suppose $k = n-1, \ldots, 2$ and $W \in \mathcal{U}_k$.
  Then for any $x_0 \in W$ and all normal domains
  $U(x_0,f|_W,r)$ with $r <  r_0 \colonequals r(x_0,f)$ (as in Lemma \ref{lemma:LemmaX})
  there exists a parameterized collection of 
  homeomorphisms 
  \begin{align*}
    h_t \colon \partial U(x_0,f|_W,r_0) \to \partial U(x_0,f|_W,t),
  \end{align*}
  $t \in (0,r_0)$ such that the mapping
  \begin{align*}
    H \colon (0, r_0) \times \partial U(x_0,f|_W,r_0) \to U(x_0, f|_W, r_0) \setminus \{ x_0 \},
\end{align*}
\begin{align*}
    H(t,x) = h_t(x) 
  \end{align*}
  is also a homeomorphism and $U(x_0,f|_W,r_0) \cong \cone(\partial U(x_0,f|_W,r_0))$.
\end{proposition}
\begin{proof}
  For $t \in (0,r_0)$ and any given point  $x \in \partial U(x_0,f|_W,t)$, we 
  define the homeomorphism $h_{t}$ to map $x$ to the endpoint of the
  unique lift, guaranteed by Proposition \ref{prop:UniqueLifts}, of the straight line segment containing
  $f(x)$ that has $f(x_0)$ as an endpoint and another endpoint in $B(f(x_0),r_0)$.  Since the lifts are unique, $h_t$ is well-defined. 
  
  Let $y \in \partial U(x_0,f|_W,t)$.  Then $f(y) \in B(f(x_0),t)$ and there is a unique lift of the straight line segment from $f(x_0)$ through $f(y)$ to $B(f(x_0),r_0)$.  The definition of $h_t$ gives that if $x$ is the endpoint of the unique lift, then $h_t(x) = y$.  So the map $h_t$ is surjective.
  It is injective since otherwise the lifts of two disjoint line segments would intersect. So there exists an 
  inverse map for $h_{t}$. Since these two maps can be both defined by lifts of line segments, it suffices
  to show that $h_t$ is continuous to prove the claim.

  Suppose that there exists a sequence $\{a_j\}_{j \in \mathbb N}$ such that $a_j \in \partial U(x_0,f|_W,r_0)$ and $\lim_{j \to \infty} a_j = a \in \partial U(x_0,f|_W,r_0)$.
  This would imply
  that there is a radial line segment $I$, connecting $f(a)$ to $f(x_0)$, together with a sequence $(I_j)$ of line segments, connecting $f(a_j)$ to $f(x_0)$,
  converging to $I$.  We must show that the unique lifts $\alpha_j$ of $I_j$ converge to the unique lift $\alpha$
  of $I$. If we consider the paths $\alpha_j$ as sets in $\R^n$, the compactness of the Hausdorff metric (see e.g.\ \cite[pp.\,70--77]{BridsonHaefliger}) implies that $\{\alpha_j\}_{j \in \mathbb N}$ must have a converging subsequence.  So by taking a subsequence suppose that $\lim_{j \to \infty} \alpha_j = \beta$.  Additionally, $\beta$ will be connected since for each $j \in \mathbb N$, $\alpha_j$ is connected.
  
  We can parametrize the $I_j$ by a time parameter $t$ in the obvious way.  Similarly, we can parametrize $\alpha_j$ so that $f\circ \alpha_j(t) = I_j(t)$.
  By \cite[Lemma 5.32]{BridsonHaefliger}, for every $x \in \beta$, there exists a sequence $\{\alpha_j(t)\}_{j\in\mathbb N}$ so that $\lim_{j \to \infty} \alpha_j(t_j) = x$.  Since $f$ is continuous and $f(\alpha_j(t_j)) \in I_j$, we have that $f(x) \in I$.  So $\beta \subset \alpha$.  Note that $\beta$ cannot be contained in a different preimage of $I$ by $f$ since $\lim_{j \to \infty} a_j = a \in \alpha$ and $\beta$ is connected.
  
  If $x \in \alpha$, then there exists a sequence of points $y_j \in I_j$ such that $\lim_{j \to \infty} y_j = f(x)$.  The point $y_j$ has a unique preimage $x_j \in \alpha_j$ for all $j \in \mathbb N$. By \cite[Lemma 5.32]{BridsonHaefliger} there exists a subsequence $\{x_{j_k}\}_{k \in \mathbb N}$ such that $\lim_{k \to \infty} x_{j_k} = x' \in \beta \subset \alpha$.  So $f(x') \in I$ and by uniqueness of lifts we have that $x = x'$.  This gives that $\beta = \alpha$.  The argument shows that every subsequence of $\{\alpha_j\}_{j \in \mathbb N}$ must limit to $\alpha$ and so $\lim_{j \to \infty} h_t(a_j) = h_t(a)$, which gives that $h_t$ is continuous.
  
  Finally, it is straightforward to check that $H$ is also a homeomorphism, which
  implies that $U(x_0,f|_W,r_0) = \cone(\partial U(x_0,f|_W,r_0))$.
\end{proof}

The previous Proposition \ref{prop:HomeomorphismFoliation} shows that we can foliate the
small punctured lower dimensional normal domains with their boundaries. Note that
this does not a priori imply that the boundaries are spheres, see again the example
in Section \ref{sec:Nemesis}.

\subsection{Boundaries of normal domains are homeomorphic to spheres}

We wish to show that the boundary of a normal domain is homeomorphic
to a sphere for an open and discrete map $f$ with $f(B_f)$ contained in an
$(n-2)$-simplicial complex. The proof is based on an inductive argument
on the dimension of the lower dimensional normal domains.
Most of the complications in the statements and proofs of the following proposition arise from the fact that we
need to study the restriction of $f$ to the boundary of a normal domain before
showing that the boundary is a manifold.

We first compute the homology groups of the boundary of a lower dimensional normal neighborhood.

\begin{lemma}\label{lemma:homologyequality}
  Fix $k \in \{2,\dots,n-1\}$.  Let $U$ be a normal neighborhood of dimension $k$ centered at a point $x \in \R^n$.
  Let also $\partial U = V \in \mathcal{U}_k$.  If $U$ is sufficiently small, then
  \begin{align*}
    H_l(V) = H_l(\bS^k)
  \end{align*}
  for $0 \le l \leq k$, where $H_l$ is the singular homology group.
\end{lemma}
\begin{proof}
  By Proposition \ref{prop:HomeomorphismFoliation},
  $U \cong \cone(V)$ and therefore $U\setminus\{x\} \cong V\times(0,1)$.
  
  If $k < n-1$, since $V \in \mathcal{U}_k$ we know that $U$ is a normal neighborhood contained in some
  $W \in \mathcal{U}_{k+1}$. By Proposition \ref{prop:HomeomorphismFoliation} there
  exists an open set containing $U$ in $W$ that is homeomorphic to $U \times (0,1)$.
  Removing the point $x \in U$ thus gives rise to a neighborhood of $U\setminus\{x\}$
  homeomorphic to $(U \setminus \{x\}) \times (0,1) \cong V \times (0,1)^2$.
  
  We can continue inductively to find an open set containing $U$ in the top
  level normal neighborhood (which is a domain in $\R^n$) that is homeomorphic
  to $U \times (0,1)^{n-k-1}$. Furthermore, $U\setminus\{x\}$ is contained in an
  open set that is homeomorphic to $V \times (0,1)^{n-k}$.
  These are now open sets in $\R^n$ and are therefore manifolds.
  Recall that $U \cong \cone(V)$ and therefore $U$ is contractible. So $U\times (0,1)^{n-k-1}$ is also contractible.  
  
  By extending $U$ to an open domain in $\R^n$ the point $x \in U$ is extended radially. Therefore $\{x\}\times (0,1)^{n-k-1} \subset U\times (0,1)^{n-k-1}$ is an $(n-k-1)$-submanifold.
  For $1 \le l < k$, consider a map 
  \begin{align*}
    \gamma \colon \bS^l \to (U\times (0,1)^{n-k-1})\setminus (\{x\} \times(0,1)^{n-k-1}).    
  \end{align*}
  Since $U \times (0,1)^{n-k-1}$ is contractible, there is a homotopy $H$ that takes $\gamma$ to a point $x' \ne x$.  The dimension of 
  $\bS^l \times (0,1)$ is $l+1$ and 
  \begin{align*}
    (l+1) +(n-k-1) < n,
  \end{align*}
  since $l < k \le n-1$.
  So we claim that the image of $H$ can be guaranteed to avoid $\{x\} \times (0,1)^{n-k-1}$.
  To prove this claim note that $H$ can be assumed to be smooth since $U\times (0,1)^{n-k-1}\setminus (\{x\} \times(0,1)^{n-k-1})$ is an open set and hence a smooth manifold.  By the compactness of the image of $H$, there exists an $\epsilon > 0$ so that the $\epsilon$-neighborhood of $H$ lies in $U\times (0,1)^{n-k-1})\setminus (\{x\} \times(0,1)^{n-k-1})$.
  Smooth functions are dense in the uniform topology. 
  Therefore there exists a smooth function $\tilde H \colon \bS^l \times [0,1] \to U\times (0,1)^{n-k-1})\setminus (\{x\} \times(0,1)^{n-k-1})$ such that $\|H - \tilde H\| < \epsilon$.
  A straight-line homotopy takes $H$ to $\tilde H$ and hence the claim is shown.
  
  We can also assume that the image of $H$ is transverse to the submanifold $\{x\} \times(0,1)^{n-k-1}$ (see \cite[Chapter 2]{guilleminpollack}).  Since the dimensions add up to a number less than $n$, the transversality condition implies that they are actually disjoint.  The entire homotopy is disjoint from the removed set and thus 
  \begin{align*}
      \pi_l((U\times (0,1)^{n-k-1})\setminus (\{x\} \times(0,1)^{n-k-1})) = 0
  \end{align*}
for $1 \le l < k$.  
 By the above argument, $\pi_l(V) = 0$ for $1\le l < k$.  The lemma now follows by the Hurewicz theorem \cite[p. 366]{Hatcher} for this index range.
 
 It remains to show that $H_k(V) = H_k(S^k)$.  We show this case by the use of the Mayer-Vietoris theorem.  Let $M = U\times (0,1)^{n-k-1}$ and let $L =\{x\} \times\mathbb R^{n-k-1}$.  Note that
 \begin{align*}
     M \setminus L = (U\times (0,1)^{n-k-1})\setminus (\{x\} \times(0,1)^{n-k-1}).
 \end{align*}
 The Mayer-Vietoris theorem implies that
 \begin{align*}
     \cdots \to  H_{k+1}(M \cup (\mathbb R^n \setminus L) )  &\to H_k(M \cap (\mathbb R^n\setminus L) ) \to H_k(M)\oplus H_k(\mathbb R^n \setminus L)\\ 
     &\to H_k(M \cup (\mathbb R^n \setminus L) ) \to \cdots
 \end{align*}
 is an exact sequence.
We have that $M$ and $ M \cup (\mathbb R^n \setminus L)$ are contractible.  Additionally,
 \begin{align*}
     M \cap (\mathbb R^n\setminus L) = M \setminus (\{x\} \times(0,1)^{n-k-1}).
 \end{align*}
   So
\begin{align*}
    0 \to H_k(M\setminus (\{x\} \times(0,1)^{n-k-1})) \to H_k(\mathbb R^n \setminus L) \to 0.
\end{align*}
This implies that
\begin{align*}
    H_k(V) &= H_k((U\times (0,1)^{n-k-1})\setminus (\{x\} \times(0,1)^{n-k-1})\\
    &\cong H_k(\mathbb R^n \setminus L) \cong H_k(S^k). 
\end{align*}
\end{proof}

We next show that the boundary of normal domains are homeomorphic to spheres.
\begin{proposition}\label{prop:boundaryspheres}
  Let $k \in \{2,\dots,n-1\}$.  If $V \in \mathcal{U}_k$, then $V \cong \bS^k$.
\end{proposition}
\begin{proof}
By Lemma \ref{lemma:homologyequality}, $H_0(V) \cong H_0(\mathbb S^k)$ and therefore $V$ is connected.

  We now continue to prove the main claim in the proposition.  
  Suppose first that $k=1$ and fix $V \in \mathcal{U}_1$. 
  We denote the restriction $f|_V \colon V \to fV$ by $g$.  By Lemma \ref{lemma:ImageIsASphere}, $gV$ is homeomorphic to a circle.
  At this level the intersection of $f(B_f)$ with $g(V)$ is at most a finite set of points.  Since $V$ is the boundary of a normal neighborhood we can therefore conclude that actually
  \begin{align*}
    g(V) \cap f(B_f) = \emptyset.
  \end{align*}
  This implies that $V \cap B_f = \emptyset$ and 
  \begin{align*}
    g \colon V \to g(V) \cong \bS^1
  \end{align*}
  is a covering map.  
  Since the map $f$ is finite-to-one in any normal domain, we see that $g$ is a finite-to-one cover of $\bS^1$.
  This implies that $V$ is homeomorphic to $\bS^1$.
  
  Suppose next that the claim holds true for some $k < n-1$ and $V \in \mathcal{U}_{k+1}$. Fix a point $x \in V$ and take
  a normal neighborhood $W$ of $x$ such that $\partial W \in \mathcal{U}_{k}$.
  By the inductive assumption $\partial W$ is homeomorphic to $\bS^k$.
  By Proposition \ref{prop:HomeomorphismFoliation},
  \begin{align*}
    W
    \cong \cone{\partial W}
    \cong \cone{\bS^k}
    \cong B^n.
  \end{align*}
The point $x$ has a neighborhood in $V$ homeomorphic to a ball and therefore $V$ is a closed $k$-manifold.
By Lemma \ref{lemma:homologyequality},
\begin{align*}
H_l(V) \cong H_l(\bS^k)
\end{align*}
for $0 \le l \leq k$.  In the proof of Lemma \ref{lemma:homologyequality}, we also showed that $\pi_l(V) = 0$ for $1 \le l \le k-1$.
Combining these we see that $V$ is a simply connected homology $k$-sphere and so
$V \cong \bS^k$ by Theorem \ref{thm:WhiteheadHomology}.
\end{proof}

\section{PL cone mappings}
\label{sec:PL}
In this section we prove Theorems \ref{thm:maintheorem} and \ref{thm:GlobalMainTheorem}. We divide the proof into a local and global part.
An open and discrete map $f \colon \bS^n \to \bS^n$ is called \emph{locally PL with respect to a simplicial decomposition $A$ of $\bS^n$} if, for all $x \in \bS^n$, there exist an open set $U \subset \overline U \subset \bS^n$ containing $x$ and an open set $U' \subset \overline U' \subset \bS^n$  with a homeomorphism
\begin{align*}
      \phi \colon \overline{U'} \to \overline{U} \subset \bS^n
  \end{align*}
such that $f \circ \phi$ is a PL mapping.  Additionally, if $\overline{U'}$ is given a simplicial decomposition defined by $f \circ \phi$, a suitable subdivision of the $k$-simplices in  $\overline{U'}$ are mapped to $k$-simplices in a subdivision of $A$.

\begin{lemma}\label{lemma:localtoglobal}
  Let $g \colon \bS^n\to \bS^n$ be an open and discrete map whose branch set is contained in an $(n-2)$-simplicial complex. Let $A$ be a simplicial decomposition of $\bS^n$ that contains $g(B_g)$ in its $(n-2)$-skeleton.  Additionally, suppose that $g$ is locally $PL$ with respect to $A$.  Under these conditions, there exists a homeomorphism
  \begin{align*}
      \Phi \colon \bS^n \to \bS^n
  \end{align*}
  such that $g \circ \Phi$ is a PL mapping and the $k$-simplices defined by $g \circ \Phi$ are mapped to $k$-simplices in $A$.
\end{lemma}
We remark that the proof here uses ideas from the proof in \cite[Lemma 5.12]{bonkmeyer}.
\begin{proof}
The strategy of the proof will be to pull back the simplicial structure $A$ by $g$.  
The set $\bS^n$ can be covered by finitely many open sets $U$ that satisfy the conditions in the definition of the local PL property of $g$.
We refine $A$ so that $g$ maps the simplices in $\overline{U'}$ to simplices in $A$, where $\phi \colon \overline{U'} \to \overline{U}$ is a homeomorphism as described above.

In the spirit of pulling back $A$ by $g$, let $B$ be the set of simplices $\sigma$ such that $g(\sigma) \in A$ and $g|_\sigma$ is a homeomorphism onto its image.
To show that this is a simplicial structure for $\bS^n$ it suffices to show that every point lies in the interior of a unique simplex and that the intersection of two simplices is a face of those simplices.

We first show that every point lies in the interior of a unique simplex.  Let $x \in \bS^n$ and $y=g(x) \in \Delta_k^o$, where $\Delta_k$ is a $k$-simplex in $B$ and $\Delta_k^o$ is the interior of $\Delta_k$.
Let $U$ be an open set containing $x$ such that there exists a homeomorphism $\phi \colon U' \to U$ satisfying that $g \circ \phi$ is a PL mapping.  By our assumption, $x' = \phi^{-1}(x)$ is contained in a simplex $D$ that is mapped by $g \circ \phi$ onto a simplex in $A$.  Since $g\circ \phi(x') = y \in \Delta_k^o$, the simplex $D$ must be a degree $k$ simplex and $x' \in D^o$.  Additionally, $D^o = (g\circ \phi)^{-1}(\Delta_k^o)\cap U'$.

The map $g \circ \phi$ is a PL branched cover.  Therefore it is locally injective on $D^o$.
So $g$ defines a covering map from the component $\tau$ of $g^{-1}(\Delta_k^o)$ containing $x$ to $\Delta_k^o$.  Since $\Delta_k$ is simply connected, $g$ is actually a homeomorphism from $\tau$ to $\Delta_k^o$.

We claim that $g$ extends to a homeomorphism from $\sigma = \overline{\tau}$ to $\Delta_k$.  It suffices to show that $g^{-1}\colon \Delta_k^o \to \sigma$ extends continuously to the boundary.  Let $\{y_n\}_{n \in \mathbb N}$ be a sequence of points such that $y_n \to y \in \partial \Delta_n$.  Then there exists a sequence of points $\{x_n\}_{n\in \mathbb N}$ such that $g(x_n)= y_n$.  Let $a$ and $b$ be accumulation points of $\{x_n\}_{n\in\mathbb N}$.
Let $a_n$ be a subsequence that converges to $a$ and $b_n$ a subsequence that converges to $b$.

By \cite[Lemma 5.15]{bonkmeyer}, for all $\epsilon > 0$, there exists $\delta > 0$ so that
\begin{align*}
    g^{-1}(B(y,\delta)) \subset \cup_{z \in g^{-1}(y)} B(z,\epsilon).
\end{align*}
By choosing $\epsilon$ sufficiently small, the sets $B(z,\epsilon)$ will be pairwise disjoint for $z \in g^{-1}(y)$.  However, for large $n$, $g(a_n)$ and $g(b_n)$ will be in $B(y,\delta)$.  If $a_n$ and $b_n$ are connected by a path $\gamma$, then $g^{-1}(\gamma)$ must be a path connecting $a_n \in B(a,\epsilon)$ and $b_n \in B(b,\epsilon)$.  So $g^{-1}(\gamma)$ lies outside $\cup_{z\in g^{-1}(y)}B(z,\epsilon)$, which gives a contradiction if $a \ne b$.
Thus $g^{-1}$ extends continuously to $\partial \Delta_k$ and $g$ defines a homeomorphism from $ \sigma$ to $\Delta_k$.  This shows that $ \sigma$ defines a $k$-simplex in $B$ and that $x \in  \sigma$.
This shows that every $x$ is in a simplex defined by $B$.

Let $\sigma_1$ and $\sigma_2$ be simplices in $B$ and suppose $\sigma_1\cap \sigma_2 \ne \emptyset$.  If $\sigma_1^o\cap \sigma_2^o \ne \emptyset$, then they must both be $k$-simplices and by construction must be mapped homeomorphically onto the same $k$-simplex $\Delta_k \in A$.  This is not possible since $\Delta_k$ is simply connected.

If $\sigma_1^o\cap \sigma_2^o = \emptyset$, then suppose $\tau$ is a simplex such that $\tau^o \cap \sigma_1\cap \sigma_2 \ne \emptyset$. 
It follows that $g(\tau) \subset g(\sigma_1)\cap g(\sigma_2)$.  The map $g$ defines an inverse on $g(\tau)^o$, which must agree with the inverses that it defines on $g(\sigma_1)^o$ and $g(\sigma_2)^o$.  So the entirety of $\tau$ must be contained in $\sigma_1\cap \sigma_2$.
This implies that $\sigma_1\cap \sigma_2$ is comprised of the union of finitely many simplices.

Finally, we claim that $A$ and $B$ can be refined so that the intersection of two simplices is a face. 
Let $\sigma_1$ and $\sigma_2$ be $k$-simplices.  Suppose that $\sigma_1\cap \sigma_2 \ne \emptyset$ and that there are two $(k-1)$-simplices whose interiors are in $\sigma_1\cap \sigma_2$.  We apply a barycentric subdivision $A$.  Then $g$ pulls back this decomposition to a refinement of $B$ and the new simplices in $\sigma_1$ and $\sigma_2$ cannot share more than one $(k-1)$-simplex. 
We may now proceed by repeated barycentric subdivision to rule out the cases when $\sigma_1\cap\sigma_2$ contain more than one interior of lower degree shared simplices.  At the end of this process, the refined $B$ must be a simplicial decomposition of $\bS^n$.

The construction implies that $g$ is a simplicial map from $S_B^n$ to $S_A^n$.
Thus there exists a PL map from $S_B^n$ to $S_A^n$ that is topologically equivalent to $g$ with respect to $A$.
\end{proof}

\begin{lemma}\label{lemma:LocallyPL}
  Let $f \colon \bS^n\to \bS^n$ be an open and discrete map with $f(B_f)$ contained in a simplicial $(n-2)$-complex.  Let $A$ be a simplicial decomposition of $\bS^n$ that contains $f(B_f)$ in its $(n-2)$-skeleton.
  Then $f$ is locally PL with respect to $A$.
\end{lemma}
\begin{proof}
  We proceed by induction on $n$. The base case, $n = 2$, follows from Sto\"ilow's theorem (see \cite{Stoilow} or \cite{LuistoPankka-Stoilow})
  as $f$ is topologically equivalent to a rational map $S^2 \to S^2$ and rational maps are topologically equivalent to PL mappings.
  
  We now suppose that $f\colon \bS^n \to \bS^n$ is defined as in the statement of the lemma.  Then there exists a Euclidean simplicial decomposition $A$ (when $\bS^n$ is viewed as $\mathbb R^n \cup \{\infty\}$) such that $f(B_f)$ is contained in the $(n-2)$-skeleton of $A$.
  
  Fix $x \in \bS^n$. For a small radius $r_0$, there exists a ball $B(f(x),r_0)$ that is radially symmetric with respect to the simplicial decomposition $A$.  More precisely, for any simplex $\Delta \in A$,
  \begin{align*}
      \Delta \cap \partial B(f(x),r) = \frac{r}{s} (\Delta \cap \partial B(f(x),s))
  \end{align*}
  for $0<r,s \le r_0$, where $r/s$ is the dilation mapping the $s$-sphere at $f(x_0)$ to the $r$-sphere.
  
  By Proposition \ref{prop:HomeomorphismFoliation} and Proposition \ref{prop:boundaryspheres}, for sufficiently small $r_0$, the normal
  neighborhood $U(x,f,r_0) \cong \cone(V)$, where $V = \partial U(x,f,r_0)$, is homeomorphic to $S^{n-1}$.
  Let $g = f|_V$.  
  By the construction of the homeomorphism in Proposition \ref{prop:HomeomorphismFoliation},
  $f$ is topologically equivalent to $\cone(g) \colon \cone(V) \to B(f(x),r_0)$.
  By the choice of $B(f(x),r_0)$, the map $g \colon V \to \partial B(f(x),r_0)$ sends its branch set into the $(n-3)$-skeleton of $B(f(x),r_0)$ induced by $A$.
  The induction hypothesis gives that $g$ is locally a PL mapping which respects the simplicial decomposition $A$. By Lemma \ref{lemma:localtoglobal} it is globally a PL mapping, which respects the simplicial decomposition $A$.
  
  The set $B(f(x),r_0)$ was chosen to be radially symmetric.
  Therefore, the map $\cone(g)$ also respects the simplicial decomposition $A$ on $\bS^n$.
  Thus $f$ satisfies the conclusion of the lemma.
\end{proof}

Theorem \ref{thm:maintheorem} follows immediately from  Lemma \ref{lemma:LocallyPL}.
The combination of Lemmas \ref{lemma:localtoglobal} and \ref{lemma:LocallyPL} proves Theorem \ref{thm:GlobalMainTheorem}.

\section{Homotopic properties of foliations}
\label{sec:HomotopiesOfHomotopies}

In Section \ref{sec:Bdry} we showed that the boundaries of normal neighborhoods locally foliate a punctured normal neighborhood.
Furthermore, when the image of the branch set is a simplicial $(n-2)$-complex, this foliation is the trivial
one, i.e., it consists only of spheres. The proof in Section \ref{sec:Bdry} relied strongly on the fact
that by Proposition \ref{prop:HomeomorphismFoliation} the boundaries are homeomorphic. This enabled us to show
that the boundaries are not only manifolds but even spheres. 

In this section we show that the existence of the homeomorphisms given by Proposition \ref{prop:HomeomorphismFoliation}
is not needed if we a priori assume the boundaries to be manifolds.
Compare these results to the example in Section \ref{sec:Nemesis}, where we noted that the boundaries of normal neighborhoods of
the double suspension map are not manifolds but do foliate a punctured domain in $\bS^5$.
For clarity we state the results here as concerning codimension 1 closed submanifolds in $\R^n$ instead of
focusing on boundaries of normal neighborhoods.
We prove that the only topological codimension 1 manifold foliations of punctured
domains in euclidean spaces are the trivial spherical ones. We have not been able to find this statement
recorded in the literature in this generality, but we do not assume it to be unknown to the specialists in the field.
See however \cite[Theorem 3.7 and Lemma 6.2]{MartioSrebro} for the three dimensional case and compare to the
Reeb Stability Theorem \cite[Theorem 2.4.1, p.\ 67]{Foliations-book} 
for a related claim in the smooth setting.
Compare also to the Perelman stability theorem in \cite{Perelman-Stability} (see also \cite{Kapovitch-PerelmanStability}) from which a similar result could be deduced in the smooth setting.

We do not assume that leaves in foliations are homeomorphic, which arises from unique path lifts in the setting of Section \ref{sec:Bdry}.  Rather, we rely here
on the fact that in manifolds with positive injectivity radius homotopy arguments can be essentially reduced to discrete homotopy.
Note that the above-mentioned Perelman stability theorem, \cite{Perelman-Stability}, requires the assumption of a lower bound to the Ricci curvature, 
and such a lower bound also gives rise to a lower bound for the injectivity radius of a closed Riemannian manifold.

\begin{definition}\label{def:foliation}
Let $\{U_t\}$, $t \in (0,1)$, be a family of compact $n$-dimensional connected manifolds with boundary contained in $\R^n$.  If $\partial U_t = X_t$, then $\mathcal{X} \colonequals \{ X_t \}, t \in (0,1)$ is a \emph{topological foliation} if the following conditions hold.
  \begin{enumerate}[{(F}1{)}]
  \item $X_t \cap X_s = \emptyset$ when $t \neq s$.
  \item $0 \in U_t \subset U_s$ for all $t \leq s$.
  \item $0 \notin X_t$ for all $t \in (0,1)$ and $\diam(X_t) \to 0$ when $t \to 0$.
  \item $U \colonequals \{ 0 \} \cup \bigcup_{t} X_t$ is an open neighborhood of the origin.
  \end{enumerate}
  By the above properties, there exists a parameter $t_0 \in (0,1)$ and a radius $r_0 > 0$ such
  that $U_{t_0} \subset B(0,r_0) \subset U$. We call the pair $(t_0,r_0)$ a \emph{break point} of the foliation.
\end{definition}
We remark that the assumption that $U_t$ is a manifold with boundary seems unnecessary to us.  We believe that instead it should suffice to assume that $X_t$ are closed, connected $(n-1)$-submanifolds in $\R^n$ along with (F1)-(F4) should suffice for the results that follow.

The aim of this section is to show that the definition above always leads to a trivial foliation.
\begin{theorem}\label{thm:SphericalityOfFoliation}
  For any topological foliation $\mathcal{X} = \{ X_t \}$ there exists
  a break point $(t_0,r_0)$ with $t_0 \in (0,1), r_0 > 0$ such that $X_t$ is a topological sphere
  for all $t \leq t_0$.
\end{theorem}

To prove this claim we will need several auxiliary results.
For the purposes of the upcoming proofs we denote
\begin{align*}
  A^\mathcal{X}(a,b)
  = \bigcup_{t \in (a,b)} X_t
\end{align*}
and call such sets \emph{foliation annuli}.
For subintervals $I$ of $(0,1)$ we similarly
use the notation
$A^\mathcal{X}(I)$.

\begin{lemma}\label{lemma:FoliatonAnnuliIsAManifold}
  Let $\mathcal{X} = \{ X_t \}_{t \in (0,1)}$ be a topological foliation. Then for any $a,b \in (0,1)$, $a<b$,
  the foliation annulus $A^\mathcal{X}(a,b)$ is an $n$-manifold.
\end{lemma}
\begin{proof}
  We show that $A^\mathcal{X}(a,b)$
  is an open subset of $\R^n$. To this end, fix
  a point $x_0 \in A^\mathcal{X}(a,b)$. 
  
  By condition (F2) in Definition \ref{def:foliation}, we have that $x_0 \in U_b \cap ( \R^n \setminus \overline{U_a})$.
  So there exists an $\eps > 0$ so that $B(x_0,\eps) \subset U_b \cap (\R^n \setminus \overline{U_a})$.  For $x \in B(x_0,\eps)$ let $t_x = \inf \{t  : x \in U_t\}$.  The set is nonempty set since $x \in U_b$.  From this definition we see that $x \in X_{t_x}$.  Additionally, $a<t_x < b$ by our choice of $\eps$.  Therefore $B(x_0,\eps) \subset A^\mathcal{X}(a,b)$.

\end{proof}

\begin{lemma}\label{lemma:HomotopyLemma}
  Let $X$ be a metric space and let $M$ be a compact topological manifold embedded in $\mathbb R^n$.  There exists a $\delta > 0$ such that for any continuous maps $f,g\colon X \to M $, if $\sup_{x \in X} \|f(x) - g(x)\| < \delta$, then $f$ and $g$ are homotopic.
\end{lemma}
\begin{proof}
  By \cite[Theorem A.7]{Hatcher}, there exists an open neighborhood $N \subset \mathbb R^n$ of $M$ such that $M$ is a retract of $N$.  If $\delta$ is sufficiently small, then for all $p \in M$, $B_\delta(p) \subset N$.  So the image of the homotopy $H(t,x) = (1-t)f(x) + tg(x)$, for $t \in (0,1)$, is contained in $N$.  Let $r \colon N \to M$ be the retraction map.  Then $r \circ H$ is a homotopy from $f$ to $g$ lying in $M$.
\end{proof}

Motivated by this lemma we fix some terminology on discrete approximations of homotopies
in the setting of manifolds.
\begin{definition}
  Let $M$ be a compact $n$-manifold and let $\delta$ be the parameter given in Lemma \ref{lemma:HomotopyLemma}.  By Lemma \ref{lemma:HomotopyLemma}, if $X$ is a metric space, then any two mappings $f,g \colon X \to M$
  with $\sup_{x \in X} \|f(x) - g(x)\| < \delta$ are homotopic. In such a setting we
  say that \emph{a discrete homotopy approximation of a continuous map} $f \colon [0,1]^k \to M$
  is a discrete collection of points $D \subset [0,1]^k$ together with a mapping
  $g \colon D \to M$ such that there is a continuation $\tilde g$ of $g$ that satisfies
  $\sup_{x \in [0,1]^k} \|f(x) - \tilde g(x)\| < \delta$.
\end{definition}

We next show that any homotopy performed in the union of a foliation $\mathcal{X}$ can be `pulled'
within one of the leaves. For the sake of clarity and readability we state
the main proposition for a general mapping instead of a homotopy.
Note that with minor modifications this argument could be used with
more general topological foliations that are not converging to a point.
\begin{proposition}\label{prop:HomotopyOfHomotopies}
  Let $\mathcal{X} = \{ X_t \}$ be a topological foliation.
  Then for any $k \in \N$ and any continuous mapping
  \begin{align*}
    f \colon \bS^k \times [0,1] \to A^\mathcal{X}(0,1)
  \end{align*}
  there exists a continuous mapping
  \begin{align*}
    g \colon \bS^k \times [0,1] \to X_{t_0}
  \end{align*}
  such that $f$ and $g$ are homotopic in $A^\mathcal{X}(0,1)$.
\end{proposition}
\begin{proof}
  Fix $k \in \N$ and note that for any mapping
  \begin{align*}
    f \colon \bS^k \times [0,1] \to A^{\mathcal{X}}(0,1)
  \end{align*}
  there is a positive distance from the image of $f$ to the boundary
  of $A^{\mathcal{X}}(0,1)$.

  We define $I$ to be a maximal, possibly trivial, subinterval of $(0,1)$ such that
  for any mapping
  \begin{align*}
    f \colon \bS^k \times [0,1] \to A^{\mathcal{X}}(I)
  \end{align*}
  there exists a mapping
  \begin{align*}
    g \colon \bS^k \times [0,1] \to X_{t_0}
  \end{align*}
  such that $f$ and $g$ are homotopic.
  We wish to show that in fact $I = (0,1)$, which will prove the claim.
  In order to show this we demonstrate that $I$ is non-empty and both
  open and closed.

  Since $\{ t_0 \} \subset I$, the interval is clearly non-empty.
  To show that $I$ is open, we fix $s_0 \in I$ and cover the $(n-1)$-dimensional manifold $X_{s_0}$
  with open sets that form an atlas. Let $\delta_{s_0}'$ be the Lebesgue number of this open cover.
  Next we cover $X_{s_0}$ with charts of the $n$-manifold $A^{\mathcal{X}}(0,1) \supset X_{s_0}$
  and denote by $\delta_{s_0}''$ the Lebesgue number of this open cover. We
  set $\delta_{s_0} = 8\inv \min\{\delta_{s_0}', \delta_{s_0}''\}$ and fix $\eps > 0$ such that
  $A^{\mathcal{X}}(s_0-\eps, s_0+\eps) \subset B^n(X_{s_0}, \delta_{s_0})$.

  Fix now a mapping
  \begin{align*}
    f \colon \bS^k \times [0,1] \to A^{\mathcal{X}}(I \cup (s_0-\eps, s_0 + \eps)).
  \end{align*}
  Since the domain of $f$ is compact, $f$ is uniformly continuous, so there
  exists $\delta > 0$ such that $\| f(x) - f(y) \| < \eps/4$ for all
  $x,y$ with $d(x,y) < \delta$.
  Let $D \subset \bS^k \times [0,1]$ be a discrete set such that
  $B(D,\delta/4) = \bS^k \times [0,1]$ and denote
  $\hat f \colonequals f|_D$.
  By our selection of $\eps$ we can now define a mapping
  \begin{align*}
    \hat g \colon D \to X_{s_0}
  \end{align*}
  such that $d_\infty(\hat f, \hat g) < \eps/4$, where $d_\infty$ denotes the supremum norm.
  
  Now, since $\eps < \delta_{s_0}'/8$, we see that the
  mapping $\hat g$ can be extended through affine continuations in the
  charts of $X_{s_0}$ into a continuous map
  \begin{align*}
    g \colon \bS^k \times [0,1] \to X_{s_0}.
  \end{align*}
  We immediately see that we also have
  $d_\infty(f,g) < \eps/2$ and
  so by the definition of $\eps$ we see
  that $f$ and $g$ are homotopic as we can
  use the affine line homotopy within
  the charts of $A^{\mathcal{X}}(0,1) \supset X_{s_0}$.
  Thus we conclude that
  $$
  (s_0 - \eps, s_0 + \eps)
  \subset I
  $$
  and so $I$ is open.
 
Finally we need to show that $I$ is closed. We may suppose that $b \in (0,1)$ is such that $(b-\epsilon,b) \subset I$ for some $\epsilon > 0$.
Suppose $f\colon \bS^k \times [0,1] \to A^{\mathcal{X}}([b-\epsilon,b])$ is a continuous map. Then $f$ maps into $U_b$. Since $U_b$ is a manifold with boundary, there exists a number $\delta' > 0$ such that we can form an atlas for $U_b$ with Lebesgue number $\delta'$.  The same argument as above gives a discrete set $D\subset S^k \times [0,1]$, with $B(D, \delta) = S^k\times [0,1]$, and a map 
\begin{align*}
    \hat g \colon D \to A^{\mathcal{X}}(b-\epsilon,\epsilon)
\end{align*}
such that $\hat g$ is uniformly close to $\hat f := f|_D$.  That is, if $d(x,y) < \delta/4$, then $\|\hat g(x) - \hat g(y)\| < \delta'/2$.
Since $U_b$ is a manifold with boundary we can extend $\hat g$ to a map
\begin{align*}
    g \colon S^k\times [0,1] \to   A^{\mathcal{X}}(b-\epsilon,b).
\end{align*}
Note that if $U_b$ were just an open set and not a manifold with boundary, it would not be clear that the image of $g$ would not intersect $X_b$.
Since $U_b$ is a manifold with boundary, the extension can be done affinely in the charts near the boundary and therefore will not intersect the boundary.

We still have that $f$ and $g$ are uniformly close and so $f$ is homotopic to $g$. 
Thus we can conclude that $b \in I$, and the proof is complete.

\end{proof}

\begin{corollary}\label{coro:Homotopyspheres}
  Let $\mathcal{X} = \{ X_t \}$ be a topological foliation with $(t_0,r_0)$ its break point. 
  Then for any $k = 1, \ldots, n-2$ and $t \leq t_0$, $\pi_k(X_{t}) = 0$.
\end{corollary}
\begin{proof}
  Fix $t < t_0$ and let $\alpha \colon \bS^k \to X_t$, $k \in 1, \ldots, n-2$.
  Denote by $\iota \colon X_t \to B(0,r_0)$ the inclusion map.
  Since $\pi_k(B(0,r_0) \setminus \{ 0 \}) = 0$, there exists a homotopy
  \begin{align*}
    H \colonequals \bS^k \times [0,1] \to B(0,r_0) \setminus \{ 0 \}
  \end{align*}
  taking $\alpha$ to the constant map $\alpha(0)$.
  By Proposition \ref{prop:HomotopyOfHomotopies} there exists a homotopy
  \begin{align*}
    \tilde H \colonequals \bS^k \times [0,1] \to X_{t}
  \end{align*}
  such that $H$ and $\tilde H$ are homotopic; especially
  the homotopy $\tilde H$ takes $\alpha$ to the constant map as a homotopy in
  $X_t$. Thus $\pi_k(X_t) = 0$.
\end{proof}

We are finally ready to prove the main result of this section.
\begin{proof}[Proof of Theorem \ref{thm:SphericalityOfFoliation}]
  As closed $(n-1)$-submanifolds of $\R^n$, the spaces $X_t$ are all orientable (see \cite[Theorem 3.26]{Hatcher})
  so $H_{n-1}(X_t) = \Z$ for all $t \in (0,1)$. On the other hand by
  Corollary \ref{coro:Homotopyspheres}, if $t \le t_0$ for $(t_0,r_0)$ a break point of $\mathcal{X}$, then
  $\pi_k(X_t) = \pi_k(\bS^{n-1})$ for $k = 0, \ldots , n-2$.
  This, combined with the Hurewicz isomorphism theorem implies that the
  spaces $X_t$ with $t \leq t_0$ are homotopy spheres, and thus topological
  spheres by Theorem \ref{thm:WhiteheadHomology}.
\end{proof}

\section{Reverse implication}
\label{sec:ReverseImplication}

A crucial step in the proof of our main result, Theorem \ref{thm:maintheorem} was to detect
that the boundaries of sufficiently small normal domains are manifolds
when the image of the branch set has a PL-structure.  Conversely the
regularity of the boundaries of normal domains is strongly connected
to the structure of both the branch set and the mapping in general. This was noted
already by Martio and Srebro in dimension three.

We begin with a simple example demonstrating that we cannot hope the PL property of $f(B_f)$ to be equivalent to the property of boundaries
of normal domains being manifolds.
\begin{example}\label{example:CW-complex}
  Let $w \colon \R^3 \to \R^3$ be the standard 2-to-1 winding
  around the $z$ axis. Denote by $h \colon \R^3 \to \R^3$
  a homeomorphism that takes the $z$-axis to the image of the function
  $t \mapsto (0,t^2 \cos(t\inv),t)$. Also define $f \colonequals w \circ h \circ w$.
  The branch set of $f$ is the $z$-axis union with the preimage by $h$ of the $z$-axis.
  So the image of the branch set of $f$ is the $z$-axis union with the image of $w \circ h$ of the $z$-axis.  This set will be homeomorphic to infinitely many
  connected circles converging to the origin and in particular it will not be an $(n-2)$-dimensional simplicial complex.
  However, the mapping $f$ has the property that the boundaries
  of sufficiently small normal neighborhoods are manifolds.
\end{example}
\begin{figure}[h] 
  \includegraphics[width=1\textwidth]{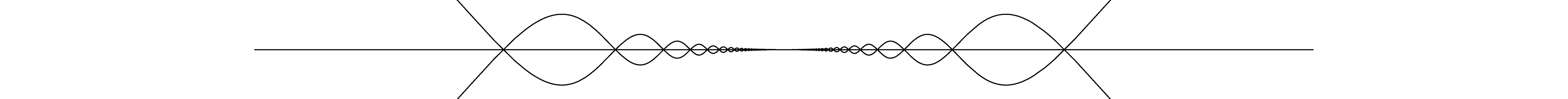}
  \caption{The image of the branch set in Example \ref{example:CW-complex}.}
  \label{fig:cosineexample}
\end{figure}

The branch set and its image in the example above do have some regularity --
even though $f(B_f)$ does not have a PL structure, it is a CW-complex. We remark that
the regularity of $f(B_f)$ being a CW-complex is not enough for our main results.
The quasiregular mappings constructed by Heinonen and Rickman in
\cite{HeinonenRickman} and \cite{HeinonenRickman2} also have CW-complex branch sets
but otherwise behave pathologically. In particular, the boundaries of normal
domains are not manifolds in those examples.

Example \ref{example:CW-complex} demonstrates that with just the assumption that the boundaries
of normal domains are manifolds we cannot deduce that the mapping is locally
a cone-type map. Instead we need to study a weaker notion of the mapping being
locally a \emph{path-type map}.
\begin{definition}\label{def:Path-TypeMap}
  Let $\Omega \subset \R^n$ be a domain and let $f \colon \Omega \to \R^n$
  be an open and discrete map. We say that $f$ is a \emph{path-type mapping} at $x_0 \in \Omega$
  or that $f$ \emph{is a path of open and discrete maps at $x_0 \in \Omega$}
  if there exists a radius $r_0 > 0$ and  a path $t \mapsto f_t$ of open and discrete maps
  $f_t \colon \bS^{n-1} \to \bS^{n-1}$ such that
  \begin{align*}
    f(x)
    = \| x_0 - x \| f_{\| x_0 - x \|}\left(\frac{x-x_0}{\| x -x_0\|}\right)
  \end{align*}
  for all $x \in B(x_0,r)$.

  We use similar terminology also when $f$ and the mappings in the path
  are quasiregular mappings.
\end{definition}

With the aid of the results in Section \ref{sec:HomotopiesOfHomotopies}
we can now prove the following proposition.
\begin{proposition}\label{prop:ReverseTopDim}
  Let $\Omega \subset \R^n$ be a domain and let $f \colon \Omega \to \R^n$
  be an open and discrete map (or a quasiregular mapping.) Suppose that
  for any $x \in \Omega$ and for all $r < r_x$ small enough $U(x,f,r)$
  is a manifold with boundary. Then for every $x_0 \in \Omega$,
  $f$ is a path of open and discrete maps (or quasiregular mappings) at $x_0$.
\end{proposition}
\begin{proof}
  For any fixed $x_0 \in \Omega$ it is immediate to see that for small enough $r_0 > 0$
  the boundaries $\partial U(x_0,f,r)$ with $r < r_0$ form a topological foliation
  since we assumed them to be manifolds.
  Thus by Theorem \ref{thm:SphericalityOfFoliation} each $\partial U(x_0,f,r)$
  is a topological sphere and we may set $f_t = f|_{\partial U(x_0,f,r)}$.
  After conjugating by homeomorphisms $(f_t)$ becomes a path of open and discrete maps
  between spheres and thus a path-type map at $x_0$.
\end{proof}

In higher dimensions it is again natural to ask about the structure
and behavior of the boundaries of lower dimensional normal domains.
\begin{lemma}\label{lemma:ReverseMidlevelLemma}
  Let $\Omega \subset \R^n$ be a domain and let $f \colon \Omega \to \R^n$
  be an open and discrete map. Suppose that for some $k\in \{2, \ldots, n-2\}$
  all the $(k+1)$-dimensional normal domains are manifolds with boundary.
  Then for any $V \in \mathcal{U}_{k+1}$ the restriction
  $f|_{\partial V}$ is locally a path-type map.
\end{lemma}
\begin{proof}
  The proof is identical to the proof of Proposition \ref{prop:ReverseTopDim}.
\end{proof}

The above lemma has a natural corollary when more `levels' of lower dimensional
normal domains are manifolds. To state the corollary we define that a mapping $f \colon \Omega \to \R^n$ is a \emph{$2$-repeated path}
at $x_0 \in \Omega$ if $f$ is a at $x_0$ a path of path-type open and discrete maps.
Likewise a mapping $f$ is a \emph{$k$-repeated path} if it is locally a path
of $(k-1)$-repeated paths.
\begin{corollary}\label{coro:RepeatedPath}
  Let $\Omega \subset \R^n$ be a domain and let $f \colon \Omega \to \R^n$
  be an open and discrete map. Suppose that for $k$ consecutive integers all the
  normal domains of those dimensions are manifolds with boundary.
  Then $f$ is an $k$-repeated path at $x_0$.
\end{corollary}

Note that for path-type maps,
since open and discrete maps are locally uniformly continuous,
we necessarily have for any $t_0$ that $f_t \to f_{t_0}$ uniformly
when $t \to t_0$. This in particular implies by basic degree theory (see \cite{Rickman-book})
that if $x_t \in B_{f_t}$ for all $t > t_0$ and $x_t \to x_0$ as $t \to t_0$,
then $x_0 \in B_{f_{t_0}}$. Similarly we see that if $x_0 \in B_{f_{t_0}}$,
then there must exist a continuous path $t \mapsto x_t \in B_{f_t}$ such that
$x_t \to x_0$ as $t \to t_0$. So if
$f$ is a path-type map at $x_0 \in B_{f}$, then $\dim_{\mathcal{T}}(B_f) \geq 1$,
and a similar conclusion holds under the assumptions of Lemma \ref{lemma:ReverseMidlevelLemma}.
Moreover, we can deduce the following:
\begin{corollary}\label{coro:DimensionEstimates}
  Let $\Omega \subset \R^n$ be a domain and let $f \colon \Omega \to \R^n$
  be an open and discrete map. Suppose that for some $k = 2, \ldots, n-2$
  all the normal domains of dimension less than or equal to $k$ are manifolds with boundary.
  Then $\dim_{\mathcal{T}}(B_f) \geq k$.
\end{corollary}

\section{Construction of a quasiregular mapping}
\label{sec:EdenMap}

Our main results, Theorem \ref{thm:maintheorem} and Theorem \ref{thm:GlobalMainTheorem}, can be used to produce examples of quasiregular mappings between manifolds.  We give one such construction in this section. 
\begin{proof}[Proof of Theorem \ref{prop:projective}]
  We first note that the
  manifold $\mathbb{CP}^1$ is homeomorphic to $\widehat{\mathbb C}$ and 
  $(\widehat{\mathbb C})^n$ is quasiregularly elliptic via e.g.\ the Alexander mapping, see \cite{Rickman-book}.
  Additionally, the composition of quasiregular mappings is still quasiregular.
  Thus in order to prove quasiregular ellipticity of $\mathbb{CP}^n$,
  it suffices to construct a quasiregular mapping $(\mathbb{CP}^1)^n \to \mathbb{CP}^n$. 
  
  We first construct an open and discrete map $f \colon (\mathbb{CP}^1)^n \to \mathbb{CP}^n$.  Consider the polynomial
  \begin{align*}
      p(u,v)=(z_1 u + w_1 v)\dots(z_n u + w_n v).
  \end{align*}
  The coefficients of each term are homogeneous polynomials in $([z_i:w_i])_{i=1}^n$,
  so in particular the coefficients define a continuous map $f \colon (\mathbb{CP}^1)^n\to \mathbb{CP}^n$.
  By the definition of the mapping, $f$ is locally injective outside the set 
  \begin{align*}
    B_f
    = \{([z_1:w_1],\dots,[z_n:w_n]) : [z_i:w_i]=[z_j:w_j] \text{ for } i \ne j \}
  \end{align*}
  and at each point $x \in B_f$, $f$ is $k$-to-$1$ for some $k = k(x) < \infty$.
  Thus $f$ is discrete.
  To see that $f$ is open, we note that away from $B_f$ the mapping is open by local injectivity and
  on the branch set $B_f$, $f$ is locally equivalent to a polynomial, and is thus an open map.
  Thus we conclude that $f$ is an open and discrete map.
  
  Again by the definition of $f$, it is clear that $B_f$ has locally a simplicial structure.
  Since $f$ is locally a polynomial, we see that $f(B_f)$ is also locally topologically equivalent
  to an $(n-2)$ simplicial complex in $\mathbb{CP}^n$. Thus by Theorem
  \ref{thm:maintheorem} $f$ is locally equivalent to a PL mapping
  and hence topologically equivalent to a quasiregular mapping. A similar
  argument as in Lemma \ref{lemma:localtoglobal} implies that there exists
  PL structures on $(\mathbb{CP}^1)^n$ and $\mathbb{CP}^n$ so that $f$ is
  equivalent to a PL map.  That is, there exists a map, $\tilde f \colon X \to Y$
  such that $X$ and $Y$ are PL manifolds and the following diagram commutes:
  \[
    \begin{tikzcd}
      (\mathbb{CP}^1)^n \arrow{r}{f} \arrow[swap]{d}{\phi} & \mathbb{CP}^n \arrow{d}{\psi} \\
      X  \arrow{r}{\tilde f} & Y
    \end{tikzcd}
  \]
  where the mappings $\phi$ and $\psi$ are homeomorphisms.
  The spaces $X$ and $Y$ have a PL structure and so they
  also have a quasiconformal structure. When the dimension
  is not $4$, that is, $n \ne 2$, by \cite{Sullivan} there
  is in fact a unique quasiconformal structure. Thus we can
  identify $X$ and $Y$ with $\times_{i=1}^n \mathbb{CP}^1$
  and $\mathbb{CP}^n$, respectively.  In the case $n=4$, a
  direct computation of the maps shows the same result.
  Thus we conclude that there exists a quasiregular mapping
  \begin{align*}
    \tilde f \colon (\mathbb{CP}^1)^n \to \mathbb{CP}^n
  \end{align*}
  and we conclude this implies that $\mathbb{CP}^n$ is quasiregularly elliptic
  for all $n \geq 2$.
\end{proof}

\begin{remark}
  In \cite{HeinonenRickman2}
  Heinonen and Rickman ask the following:
  \emph{Let $f \colon \bS^3 \to \bS^3$ be an open and discrete map. Do there exist homeomorphisms $h_1,h_2 \colon \bS^3 \to \bS^3$ such that $h_1 \circ f \circ h_2$ is a quasiregular mapping?}
  The methods in this section offer an advance in the understanding of the problem;
  indeed, the techniques here can be used to show that for
  $n \geq 4$ any open and discrete map $f \colon \bS^n \to \bS^n$ with 
  $f B_f$ contained in a simplicial $(n-2)$-complex is, up to a conjugation by homeomorphisms,
  a quasiregular mapping.
\end{remark}

\bigskip

\begin{center}
  \textbf{Acknowledgments.}   
\end{center}

The project started at the 2017 Rolf Nevanlinna Colloquium, and the authors would like to thank
the event and its organizers for an inspiring atmosphere.

A major part of the proofs were done while the first named author was visiting
UCLA and he would like to extend his gratitude both to the university for their
hospitality and to the second named author for accommodation.

Both of the authors extend their gratitude to Mario Bonk and Pekka Pankka for discussions on the topic and 
to Mike Miller who has patiently answered
their questions about algebraic topology.  The authors also thank the referees who provided helpful comments and revisions.

\def\cprime{$'$}\def\cprime{$'$}

\end{document}